\documentclass[11pt]{article}
\usepackage{srcltx}
\usepackage{xcolor}
\usepackage{graphicx}
\usepackage{amsmath}
\usepackage{amssymb}
\usepackage{theorem}
\usepackage{euscript}
\usepackage{epic,eepic}
\usepackage{pstricks}

\usepackage{parskip}
\usepackage{tikz,pgfplots}
\usepackage{multirow}
\usepackage{float}
\usepackage{multirow, array}
\usepackage{booktabs}
\usepackage{mathrsfs}
\usepackage{mathdots}
\usepackage{stackrel}
\definecolor{labelkey}{rgb}{0,0.08,0.45}
\definecolor{refkey}{rgb}{0,0.6,0.0}
\definecolor{Brown}{rgb}{0.45,0.0,0.05}
\definecolor{dgreen}{rgb}{0.00,0.49,0.00}
\definecolor{dblue}{rgb}{0,0.08,0.75}
\title{\sffamily A projected primal-dual 
splitting for solving constrained monotone
inclusions
\footnote{Contact author: 
L. M. Brice\~no-Arias, {\ttfamily luis.briceno@usm.cl},
}}
\author{Luis M. Brice\~{n}o-Arias and Sergio L\'opez Rivera$^1$ 
\\[5mm]
\small $\!^1$Universidad T\'ecnica Federico Santa Mar\'ia\\
\small Departamento de Matem\'atica\\
}

\newcommand{\Scal}[2]{{\bigg\langle{{#1}\:\bigg |~{#2}}\bigg\rangle}}
\newcommand{\scal}[2]{{\left\langle{{#1}\mid{#2}}\right\rangle}}
 
\newcommand{\Menge}[2]{\Big\{{#1}~\big |~{#2}\Big\}}

\newcommand{\HH}{\ensuremath{{\mathcal H}}}

\newcommand{\GG}{\ensuremath{{\mathcal G}}}

\newcommand{\Id}{\ensuremath{\operatorname{Id}}\,}
\newcommand{\RR}{\ensuremath{\mathbb{R}}}

\newcommand{\ran}{\ensuremath{\operatorname{ran}}}

\newcommand{\RP}{\ensuremath{\left[0,+\infty\right[}}

\newcommand{\RPX}{\ensuremath{\left[0,+\infty\right]}}

\newcommand{\NN}{\ensuremath{\mathbb N}}

\newcommand{\exi}{\ensuremath{\exists\,}}

\newcommand{\weakly}{\ensuremath{\:\rightharpoonup\:}}
\newcommand{\argmin}{\ensuremath{\operatorname{argmin}}}

\newcommand{\Fix}{\ensuremath{\operatorname{Fix}}}
\newcommand{\dom}{\ensuremath{\operatorname{dom}}}
\newcommand{\prox}{\ensuremath{\operatorname{prox}}}

\newcommand{\infconv}{\ensuremath{\mbox{\small$\,\square\,$}}}

\newtheorem{theorem}{Theorem}[section]

\newtheorem{corollary}[theorem]{Corollary}

\theoremstyle{plain}{\theorembodyfont{\rmfamily}%
}
\theoremstyle{plain}{\theorembodyfont{\rmfamily}%
}
\theoremstyle{plain}{\theorembodyfont{\rmfamily}%
}
\theoremstyle{plain}{\theorembodyfont{\rmfamily}%
\newtheorem{remark}[theorem]{Remark}}
\theoremstyle{plain}{\theorembodyfont{\rmfamily}%
}
\theoremstyle{plain}{\theorembodyfont{\rmfamily}%
\newtheorem{problem}[theorem]{Problem}}

\numberwithin{equation}{section}
\begin{document}
\maketitle

\begin{abstract}
In this paper we provide an algorithm for solving constrained composite 
primal-dual monotone inclusions, i.e., monotone inclusions in which a 
priori information on primal-dual solutions is represented via closed convex sets.
The proposed algorithm incorporates a projection step onto the a priori information sets 
and generalizes the method proposed in \cite{2.vu}. Moreover, under the presence of 
strong monotonicity, we derive an accelerated scheme inspired on \cite{3.CP} applied 
to the more general context of constrained monotone inclusions. In the particular case of 
convex 
optimization, our algorithm generalizes the methods proposed in 
\cite{1.condat,3.CP} allowing a priori information on solutions and we provide an accelerated 
scheme under strong convexity. An application of our approach with a priori information
is constrained convex optimization problems, in which available 
primal-dual methods impose constraints via Lagrange multiplier updates, usually 
leading to slow algorithms with unfeasible primal iterates. The proposed modification forces 
primal iterates to satisfy a selection of constraints onto which we can project, obtaining a 
faster method as numerical examples exhibit. The obtained results 
extend and improve several results in \cite{1.condat,2.vu,3.CP}.
\end{abstract}

{\bfseries Keywords:} accelerated schemes, constrained convex optimization, monotone 
operator theory,  proximity operator, splitting algorithms
\newpage
\section{Introduction}

This paper is devoted to the numerical resolution of composite primal-dual monotone 
inclusions in which a priori 
information on solutions is known. The relevance of monotone inclusions and convex 
optimization is justified via the increasing number of applications in several fields of
engineering and applied mathematics as image processing, evolution inclusions, 
variational inequalities, learning, partial differential equations, Mean Field Games, among 
others (see, e.g.,
\cite{4.Sicon1,5.Atto08,6.Atto17,7.BAKS,8.Facc03,9.Gabay83,10.Mercier79} and references 
therein).
The a priori information on primal-dual solutions is represented via closed convex 
sets in primal and dual spaces, following some ideas developed in 
\cite{11.Tsen00,12.Siopt2}. We force primal-dual iterates to belong to these information sets 
by adding additional projections on primal-dual iterates in each iteration of our proposed 
method. 

An important instance, in which the advantage of our formulation arises, is 
composite convex optimization with affine linear equality constraints. In this context, the 
primal-dual methods proposed in 
\cite{1.condat,2.vu,3.CP,13.Esser10,13.HeYuan,13.Cpock16,13.LorPock15} impose feasibility 
through Lagrange multiplier updates. A 
disadvantage of this approach is that such algorithms are usually slow and their primal 
iterates do not necessarily satisfy any of the constraints (see, e.g., \cite{7.BAKS}), leading to 
unfeasible approximate primal solutions. By projecting onto the affine subspace generated 
by the constraints, previous problem is solved. However, in several applications this 
projection
is not easy to compute because of singularity or bad conditioning on the linear system 
(see, e.g. \cite{13.CEMRACS}). 
In this context, the a priori information on primal solutions can be set as
any selection of the affine linear constraints. Indeed, since any solution is feasible, we know
it must satisfy any selection of the constraints. 
Even if in the previous context the
formulation with a priori information may be seen as artificial, in a practical point of view
this formulation allows us to propose a method with an additional projection onto an 
arbitrary selection of the constraints, which improves the efficiency of the method (see 
Section~\ref{sec:numeric}). This method forces primal iterates to satisfy 
the selection of the constraints, which can be chosen in order to compute the projection 
easily.

In this paper we provide a new projected primal-dual splitting method for solving 
constrained monotone inclusions, i.e., inclusions in which we count on a priori information 
on primal-dual solutions. We also provide an accelerated scheme of our method in the 
presence of strong monotonicity and we derive linear convergence in the fully strongly 
monotone case. In the case without a priori information, our results give an accelerated 
scheme of the method proposed in \cite{2.vu} for strongly monotone inclusions.
In the context of convex optimization, our method generalize the 
algorithms proposed in \cite{1.condat,3.CP} and \cite{13.LorPock15} without inertia, by 
incorporating a projection onto an the a priori primal-dual information set. This method is 
applied in the context of convex optimization 
with equality constraints, when the a priori information set is chosen as a selection of the 
affine linear constraints in which it is easy to project. The advantages of this approach with 
respect to classical primal-dual approaches are justified via 
numerical examples. Our acceleration scheme in the convex optimization context is 
obtained as a generalization of \cite{3.CP}, complementing the ergodic rates obtained in the 
case without projection in \cite{13.Cpock16} and, as far as we 
know, have not been developed in the literature and are interesting in their own right.

The paper is organized as follows. In Section~\ref{sec2} we set our notation and we give a 
brief background. In section \ref{sec3}, we set the constrained primal-dual monotone 
inclusion and we propose our algorithm 
together with the main results. We also provide connections with existing methods in the 
literature. In Section~\ref{sec6}, we apply previous 
results to convex optimization problems with 
equality affine linear constraints, together with numerical experiences 
illustrating the improvement in the efficiency of the algorithm with the additional projection.
We finish with some conclusions in Section~\ref{sec5}.

\section{Notation and preliminaries}
\label{sec2}
Let $\cal{H}$ and ${\cal{G}}$ be real Hilbert spaces. We denote the scalar products of 
${\cal{H}}$ and $\cal{G}$ by $\scal{\cdot}{\cdot}$ and the associated
norms by $\|\cdot\|$. The
projector operator onto a nonempty closed convex set $C\subset {\cal{H}}$ is denoted by 
$P_{C}$ and, for a set-valued operator $M: {\cal{H}} \rightarrow 2^{{\cal{H}}}$ we use 
$\mbox{ran}(M)$ for the range 
of $M$, $\mbox{gra}(M)$ for its graph, $M^{-1}$ for its inverse, $J_{M}=(\Id+M)^{-1}$
for its resolvent, and $\infconv$ stands for the parallel sum 
as in \cite{14.Livre1}. Moreover, $M$ is $\rho$-strongly monotone if, for 
every $(x,u)$ and $(y,v)$ in
$\mbox{gra}(M)$, $\left\langle x-y, u-v\right\rangle \geq \rho\left\|x-y\right\|^{2},$
it is $\rho-$cocoercive if $M^{-1}$ is $\rho-$strongly monotone,
$M$ is monotone if it is $\rho$-strongly monotone with $\rho=0$,
and it is maximally monotone if its graph is maximal in the sens of inclusions in 
$\HH\times\HH $, among the graphs of monotone operators. 
The class of all lower semicontinuous convex functions $f : {\cal{H}} \rightarrow (-\infty, 
+\infty]$
such that $\dom(f)=\left\{x\in {\cal{H}} \,|\, f(x)<+\infty \right\} \neq \varnothing$ is denoted 
by $\Gamma_{0}({\cal{H}})$ and, for every $f\in \Gamma_{0}({\cal{H}})$, the Fenchel 
conjugate of $f$ is denoted by $f^{*}$, its subdifferential by $\partial f$, and its 
proximity operator by $\prox_f$, as in \cite{14.Livre1}. 
We recall that $(\partial f)^{-1}=\partial f^{*}$ and $J_{\partial f} = \prox_{f}$. In addition, 
when 
$C\subset {\cal{H}}$ is a convex closed subset, we have that $J_{\partial 
	\iota_{C}}=\prox_{\iota_C}=P_{C}$, where $\iota_C$ is the indicator function of $C$, which 
	is 
$0$ in $C$ and $+\infty$ otherwise. Given $\alpha\in]0,1[$, an operator 
$T\colon\HH\to\HH$ satisfying $\Fix T\neq\varnothing$ is $\alpha-$averaged 
quasi-nonexpansive if, for every $x\in\HH$ and $y\in\Fix T$ we have $\|Tx-y\|^2\le 
\|x-y\|^2-(\frac{1-\alpha}{\alpha})\|x-Tx\|^2$.
We refer the reader to \cite{14.Livre1} for definitions and further results 
in monotone operator theory and convex optimization.
\section{Problem and main results}
\label{sec3}
We consider the following problem.
\begin{problem}
	\label{prob:main}
	Let $T\colon\HH\to\HH$ be an $\alpha-$averaged quasi-nonexpansive operator with 
	$\alpha\in]0,1[$, 
	let $V$ be a closed vector subspace of $\GG$, let $L\colon\HH\to\GG$ be a nonzero 
	linear 
	bounded 
	operator satisfying $\ran L\subset V$, let $A:{\cal{H}}\rightarrow 2^{\cal{H}}$ and
	$D\colon{\GG}\rightarrow 2^{\GG}$ 
	be maximally monotone operators which are $\rho$ and 
	$\delta-$strongly 
	monotone, respectively, and let $B:{\GG}\rightarrow 2^{\GG}$ and $C\colon\HH\to\HH$ 
	be $\chi$ and $\beta-$cocoercive, 
	respectively,
	for $(\rho,\chi)\in\RP^2$ and  $(\delta,\beta)\in\RPX^2$. 
	The problem is to solve the primal and dual inclusions
	\begin{align}\label{inc:primal}\tag{$\mathcal{P}$}
		&\text{find }\quad \hat{x}\in\Fix T\quad\text{such that}\quad 0\in A\hat{x}+
		L^*(B\infconv D)(L\hat{x})+C\hat{x}\\
		\label{inc:dual}\tag{$\mathcal{D}$}
		&\text{find }\quad \hat{u}\in V\quad\text{such that}\quad \left(\exi\hat{x}\in 
		\Fix T\right)\quad \begin{cases}
			-L^*\hat{u}\in A\hat{x}+C\hat{x}\\
			\hat{u}\in (B\infconv D)(L\hat{x}),
		\end{cases}
	\end{align}
	under the assumption that solutions exist. 
\end{problem}
When $A=\partial f$, $B=\partial g$, $C=\nabla h$, and $D=\partial \ell$, where 
$f\in\Gamma_0(\HH)$, $g\in\Gamma_0(\GG)$, 
$h\colon\HH\to\RR$ is a differentiable 
convex function with $\beta^{-1}-$Lipschitz gradient, and $\ell\in\Gamma_0(\GG)$ is 
$\delta-$strongly convex,
Problem~\ref{prob:main} reduces to
\begin{equation}
	\label{e:primal}\tag{$\mathcal{P}_0$}
	\text{find}\quad \hat{x}\in\Fix T\cap\argmin_{x\in 
	{\cal{H}}}F(x):=f(x)+(g\infconv\ell)(Lx)+h(x)
\end{equation}
together with the dual problem
\begin{equation}
	\label{e:dual}\tag{$\mathcal{D}_0$}
	\text{find}\quad \hat{u}\in V\cap\argmin_{u\in\GG}g^*(u)+(f^*\infconv h^*)(-L^*u)+\ell^*(u),
\end{equation} 
assuming that some qualification condition holds.
Note that, when $T=P_X$, any solution to \eqref{e:primal} 
is a solution to $\min_{x\in X}F(x)$, but
the converse is not true. The set $X$ in this case represents an a priori information on the 
primal solution. As you can see in the next section, an application of this formulation 
is constrained convex optimization, in which $X$ may represent a selection of the affine 
linear constraints. Even if, in this case, the formulation can be set without considering the 
set $X$, its artificial appearance has a practical relevance: the method obtained include a 
projection onto $X$ which helps to the performance of the method as stated in 
Section~\ref{sec6}. 

When $\rho=\chi=0$, $V=\GG$ and $T= \Id$, 
\eqref{e:primal}-\eqref{e:dual} 
can be solved by using \cite[Theorem~4.2]{16.CombPes12} or 
\cite[Theorem~5]{13.LorPock15}. In the last method, inertial 
terms are also included.
In the case when $\ell^*=0$ the algorithm in 
\cite{1.condat} can be used and if $\ell^*=h=0$, \eqref{e:primal}-\eqref{e:dual} can be 
solved by \cite{3.CP,15.Siopt1} or a version of \cite{3.CP} with linesearch proposed in 
\cite{19.MaliP18}. In \cite{3.CP}, the strong convexity is exploited via acceleration schemes.
Moreover, when $T=P_X$, $X\subset\HH$ is nonempty, closed and convex, $V=\GG$ and 
$\ell^*=h=0$,
\eqref{e:primal}-\eqref{e:dual} is solved in \cite[Theorem~3.1]{7.BAKS}. 
When $\rho>0$ or $\chi>0$, ergodic 
convergence rates are derived in \cite{13.Cpock16} when $V=\GG$ and $T= \Id$.
In its whole generality, as far as we know, \eqref{e:primal}-\eqref{e:dual} has not been 
solved and strong convexity has not been exploited.

In Problem~\ref{prob:main} set $T=\Id$, 
$V=\GG=G_1\oplus\cdots\oplus G_m$, $L\colon 
x\mapsto(L_1x,\ldots,L_mx)$, $B\colon (u_1,\ldots,u_m)\mapsto \omega_1B_1u_1\times 
\cdots\times\omega_mB_mu_m$ and $D\colon (u_1,\ldots,u_m)\mapsto 
\omega_1D_1u_1\times \cdots\times\omega_mD_mu_m$, where for every $i\in\{1,\ldots,m\}$, 
$L_i\colon\HH\to G_i$ is linear and bounded, $B_i\colon G_i\mapsto 2^{G_i}$ and $D_i\colon 
G_i\mapsto 2^{G_i}$ are maximally monotone operators such that $D_i$ is strongly 
monotone, and $\omega_i>0$ satisfies $\sum_{i=1}^m\omega_i=1$. Then, 
Problem~\ref{prob:main} reduces to \cite[Problem~1.1]{16.CombPes12} (see also 
\cite[Problem~1.1]{2.vu}). We prefer to set $m=1$ for simplicity. In \cite{16.CombPes12}, 
previous problem is solved 
when $C$ is monotone and Lipschitz by applying the method in \cite{17.Tsen00} to the 
product primal-dual space. Accelerated 
versions of previous algorithm under strong monotonicity are proposed in \cite{18.bot14}. 
The cocoercivity of $C$ is exploited in \cite{2.vu}, where an algorithm is proposed for 
solving 
Problem~\ref{prob:main} when $\rho=\chi=0$, $V=\GG$ and $T= \Id$. 
In the following theorem we provide an algorithm for solving Problem~\ref{prob:main} in its 
whole generality with weak convergence to a solution when the stepsizes are fixed. 
Moreover, when $A$ or $B^{-1}$ are strongly monotone ($\rho>0$ or $\chi>0$), we provide 
an accelerated version inspired on (and generalizing) \cite[Section~5.1]{3.CP}. Finally, 
we generalize
\cite[Section~5.2]{3.CP} for obtaining linear convergence when $\rho>0$ and $\chi>0$.
\begin{theorem}
	\label{thm:main}
	Let $\gamma_{0}\in]0,2\delta[$ and $\tau_{0}\in]0,2\beta[$ be such that 
	\begin{equation}
		\label{e:parcond}
		\|L\|^2\le
		\left(\frac{1}{\tau_0}-\frac{1}{2\beta}\right)\left(\frac{1}{\gamma_0}-\frac{1}{2\delta}\right)
	\end{equation}
	and let $(x^0,\bar{x}^0,u^{0})\in \HH\times\HH\times\GG$ such that $\bar{x}^{0}=x^{0}$. 
	Let 
	$(\theta_k)_{k\in\NN}$, $(\gamma_k)_{k\in\NN}$ and $(\tau_k)_{k\in\NN}$ be sequences 
	in $]0,1]$, $]0,2\delta[$ and $]0,2\beta[$, respectively, and consider 
	\begin{equation}
		\label{e:alg2}
		(\forall k\in \mathbb{N})\quad 
		\left\lfloor 
		\begin{array}{ll}
			\eta^{k+1}=J_{\gamma_{k} B^{-1}}(u^k+\gamma_{k} (L\bar{x}^k-D^{-1}u^k))\\
			u^{k+1}=P_V\,\eta^{k+1}\\
			p^{k+1}=J_{\tau_{k} A}(x^k-\tau_{k} (L^*u^{k+1}+Cx^k))\\
			x^{k+1}=T\,p^{k+1}\\
			\bar{x}^{k+1}=x^{k+1}+\theta_{k}(p^{k+1}-x^{k}).
		\end{array}
		\right. 
	\end{equation}
	Then, the following hold.
	\begin{enumerate}
		\item\label{thm:maini} For every $k\in\NN$ and for every solution $(\hat{x},\hat{u})$ 
		to Problem~\ref{prob:main}, we have
		\begin{align}
			\label{e:ineqmain}
			\hspace{-.5cm} 
			\frac{\|x^k-\hat{x}\|^2}{\tau_{k}} +\frac{\|u^k-\hat{u}\|^2}{\gamma_{k}} ∞
			&\geq (2\rho\tau_k+1)\frac{\|p^{k+1}-\hat{x}\|^2}{\tau_k}
			+\left\|p^{k+1}-x^k\right\|^2\left(\frac{1}{\tau_{k}}-\frac{1}{2\beta}\right)\nonumber\\ 
			&+(2\chi\gamma_k+1)\frac{\|\eta^{k+1}-\hat{u}\|^2}{\gamma_k} + 
			\|\eta^{k+1}-u^k\|^2\left(\frac{1}{\gamma_{k}}-\frac{1}{2\delta}\right) \nonumber\\
			&	\hspace{-1.3cm}+2\scal{L(p^{k+1}-{x}^k)}{\eta^{k+1}-\hat{u}} 
			-2\theta_{k-1}\scal{L(p^k-x^{k-1})}{\eta^{k}-\hat{u}}		\nonumber\\
			&\hspace{-1.3cm}
			-2\theta_{k-1}\|L\|\|p^k-x^{k-1}\|\|\eta^{k+1}-u^k\|.
		\end{align}
		
		\item\label{thm:mainii} Suppose that $\rho=0$ and $\chi=0$. If we set 
		$\theta_k\equiv 
		1$, $\tau_k\equiv\tau$, $\gamma_k\equiv\gamma$ and we 
		assume that \eqref{e:parcond} holds with strict inequality, 
		we obtain $x^{k}\weakly\hat{x}$ and $u^{k}\weakly\hat{u}$, for some 
		solution $(\hat{x},\hat{u})$ to Problem~\ref{prob:main}.
		\item\label{thm:mainiii} Suppose that $\rho>0$, $\chi=0$, and $D^{-1}=0$. If we set
		\begin{equation}
			\label{e:defparam}
			(\forall k\in\NN)\quad \theta_k=\frac{1}{\sqrt{1+2\rho\tau_k}},\quad 
			\tau_{k+1}=\theta_{k} 
			\tau_{k},\quad \gamma_{k+1}={\gamma_{k}}/{\theta_{k}},
		\end{equation}
		and we 
		assume that \eqref{e:parcond} holds with equality, we obtain, for every solution 
		$(\hat{x},\hat{u})$ to Problem~\ref{prob:main},
		$(\forall \varepsilon>0)(\exists 
		N_{0}\in \mathbb{N})(\forall k\geq N_{0})$\
		\begin{equation} \big\|x^{k}-\hat{x}\big\|^{2}\leq 
			\dfrac{1+\varepsilon}{k^2}\left(\dfrac{\left\|x^{0}-\hat{x}\right\|^{2}}{\rho^2\tau_{0}^{2}}
			 + 
			\frac{2\beta\|L\|^{2}}{\rho^2(2\beta-{\tau_0})} 
			\left\|u^{0}-\hat{u}\right\|^{2}\right). 
		\end{equation} 
		\item\label{thm:mainiv} Suppose that
		$\rho>0$ and $\chi>0$ and define
		\begin{equation}
			\label{e:defmualpha}
			\mu = \frac{2\sqrt{\rho\chi}}{\left\|L\right\|}\quad \text{and}\quad 
			\alpha =\min\left\{\frac{\mu \rho}{\rho+\frac{\mu}{4\beta}} , \frac{\mu 
				\chi}{\chi+\frac{\mu}{4\delta}}\right\}.
		\end{equation}
		If we set $\theta_k\equiv\theta\in((1+\alpha)^{-1}, 1]$, $\tau_{k}\equiv\tau$ and 
		$\gamma_{k}\equiv\gamma$ with
		\begin{equation}
			\label{e:deftaugamma}
			\tau = \dfrac{2\beta\mu}{\mu + 
				4\beta\rho}
			\quad\text{and}\quad  \gamma = 
			\dfrac{2\mu\delta}{\mu+4\delta\chi}, 
		\end{equation} 
		we obtain linear convergence. That is, for every $k\in \mathbb{N}$,
		\begin{multline}
			\label{clec}
			\left(\chi(1-\omega) + \dfrac{\mu}{4\delta}\right)\left\|u^{k}-\hat{u}\right\|^{2} + 
			\left(\rho + 
			\dfrac{\mu}{4\beta}\right)\left\|x^{k}-\hat{x}\right\|^{2}\\
			\leq\omega^{k}\left(\left(\chi + \dfrac{\mu}{4\delta}\right)\left\|u^{0}-\hat{u}\right\|^{2} 
			+ 
			\left(\rho + \dfrac{\mu}{4\beta}\right)\left\|x^{0}-\hat{x}\right\|^{2}\right),
		\end{multline}
		where $\omega =(1+\theta)/(2+\alpha)\in [(1+\alpha)^{-1},\theta).$
	\end{enumerate}
\end{theorem}
\begin{proof}
	\ref{thm:maini}:
	Fix $k\in\NN$ and let $(\hat{x},\hat{u})$ be a solution to Problem~\ref{prob:main}. We 
	have
	$\hat{x}\in\Fix T$, $\hat{u}\in V$ and, using $B\infconv D=(B^{-1}+D^{-1})^{-1}$, we 
	deduce $-(L^*\hat{u}+C\hat{x})\in A\hat{x}$ and $L\hat{x}-D^{-1}\hat{u}\in
	B^{-1}\hat{u}$. 
	Therefore,
	since \eqref{e:alg2} yields
	\begin{equation}
		\label{e:algoinc}
		\begin{cases}
			\frac{x^k-p^{k+1}}{\tau_{k}} - L^{*}u^{k+1} -Cx^k\in Ap^{k+1}\\
			\frac{u^k -\eta^{k+1}}{\gamma_{k}} + L\bar{x}^{k} -D^{-1}u^k\in B^{-1}\eta^{k+1}
		\end{cases}
	\end{equation}
	and $A$ and $B^{-1}$ are $\rho$ and $\chi$-strongly monotone, respectively, we deduce
	\begin{multline}
		\label{e:cocoer}
		\!\!\Scal{\dfrac{x^k\!-\!p^{k+1}}{\tau_{k}} - L^{*}(u^{k+1}\!-\hat{u})}{\!p^{k+1}\!-\hat{x}\!}+
		\Scal{\dfrac{u^k\!-\eta^{k+1}}{\gamma_{k}}\!+\!L(\bar{x}^{k} - 
			\hat{x})}{\!\eta^{k+1}\!-\hat{u}\!}\\
		-\scal{Cx^k-C\hat{x}}{p^{k+1}-\hat{x}}-
		\scal{D^{-1}u^k-D^{-1}\hat{u}}{\eta^{k+1}-\hat{u}}\\
		\ge\rho\left\|p^{k+1}-\hat{x}\right\|^2+\chi\|\eta^{k+1}-\hat{u}\|^2.
	\end{multline}
	From the cocoercivity of $C$ and $D^{-1}$ we have from $ab\le \beta 
	a^2+b^2/(4\beta)$ that
	\begin{align}
		\label{e:parteC}
		\scal{Cx^k-C\hat{x}}{p^{k+1}-\hat{x}}&=\scal{Cx^k-C\hat{x}}{p^{k+1}-{x}^k}+
		\scal{Cx^k-C\hat{x}}{x^{k}-\hat{x}}\nonumber\\
		&\ge-\|Cx^k-C\hat{x}\|\|p^{k+1}-{x}^k\|+\beta\|Cx^k-C\hat{x}\|^2\nonumber\\
		&\ge -\frac{\|p^{k+1}-{x}^k\|^2}{4\beta},
	\end{align}
	and, analogously, $	\scal{D^{-1}u^k-D^{-1}\hat{u}}{\eta^{k+1}-\hat{u}}\ge 
	-\frac{\|\eta^{k+1}-{u}^k\|^2}{4\delta}$.
	Hence, by using 
	\cite[Lemma~2.12(i)]{14.Livre1} in \eqref{e:cocoer}, we deduce
	\begin{align}
		\label{e:ec4}
		\frac{\|x^k\!-\hat{x}\|^2}{\tau_{k}}+\frac{\|u^k\!-\hat{u}\|^2}{\gamma_{k}}
		&\geq \left(\!2\rho+\!\frac{1}{\tau_k}\!\right)\left\|p^{k+1}-\hat{x}\right\|^2 
		+\left(\!2\chi+\!\frac{1}{\gamma_k}\!\right)\|\eta^{k+1}-\hat{u}\|^2\nonumber\\
		&\hspace{-.5cm}+ 2\left[\left\langle 
		L(p^{k+1}-\hat{x}) \,|\, 
		u^{k+1}-\hat{u}\right\rangle-\left\langle L(\bar{x}^k -\hat{x}) \,|\, 
		\eta^{k+1}-\hat{u}\right\rangle\right]\nonumber\\
		&\hspace{-.5cm}
		+\|\eta^{k+1}\!-{u}^k\|^2\left(\!\frac{1}{\gamma_{k}}-\frac{1}{2\delta}\!\right)
		+\|p^{k+1}\!-x^k\|^2\left(\!\frac{1}{\tau_{k}}-\frac{1}{2\beta}\!\right).
	\end{align}
	Moreover, \eqref{e:alg2}, $\mbox{ran}(L)\subset V$ and $u^k-\eta^k\in V^{\bot}$, for 
	every $k\in\NN$, yield
	\begin{align}
		\label{e:parteL}
		\scal{L(p^{k+1}-\hat{x})}{u^{k+1}-\hat{u}}
		-\scal{L(\bar{x}^k -\hat{x})}{\eta^{k+1}-\hat{u}}&\nonumber\\
		&\hspace{-5.8cm}= \scal{L(p^{k+1}-\hat{x})}{u^{k+1}-\hat{u}}-\scal{L({x}^k 
			-\hat{x})}{\eta^{k+1}-\hat{u}}
		\nonumber\\
		&\hspace{-5.4cm}-\theta_{k-1}\scal{L(p^k-x^{k-1})}{\eta^{k+1}-\hat{u}}\nonumber\\
		&\hspace{-5.8cm}= \scal{L(p^{k+1}-{x}^k)}{\eta^{k+1}-\hat{u}}
		-\theta_{k-1}\scal{L(p^k-x^{k-1})}{\eta^{k+1}-\hat{u}}\nonumber\\
		&\hspace{-5.8cm}= \scal{L(p^{k+1}-{x}^k)}{\eta^{k+1}-\hat{u}}
		-\theta_{k-1}\scal{L(p^k-x^{k-1})}{\eta^{k+1}-u^k}\nonumber\\
		&\hspace{-5.4cm}-\theta_{k-1}\scal{L(p^k-x^{k-1})}{\eta^{k}-\hat{u}}\nonumber\\
		&\hspace{-5.8cm}\ge 
		\scal{L(p^{k+1}-{x}^k)}{\eta^{k+1}-\hat{u}}-\theta_{k-1}\|L\|\|p^k-x^{k-1}\|\|\eta^{k+1}-u^k\|\nonumber\\
		&\hspace{-5.3cm}-\theta_{k-1}\scal{L(p^k-x^{k-1})}{\eta^{k}-\hat{u}},
	\end{align}
	which, together with \eqref{e:ec4}, yield \eqref{e:ineqmain}. 
	
	\ref{thm:mainii}: For every $k\in\NN$,
	it follows from Theorem~\ref{thm:main}\eqref{thm:maini}, \cite[Lemma~2.1]{20.Opti04},
	$\rho=\chi=0$, $\theta_k\equiv 1$, 
	$\tau_k\equiv\tau$, $\gamma_k\equiv\gamma$, and the properties of $T$ and $P_X$ that
	\begin{align}
		\label{e:aux1}
		\frac{\|p^k-\hat{x}\|^2}{\tau}+	\frac{\|\eta^k-\hat{u}\|^2}{\gamma}
		&\geq \frac{\|p^{k+1}-\hat{x}\|^2}{\tau}
		+\left(\frac{1-\alpha}
		{\alpha}\right)\frac{\|x^k-p^k\|^2}{\tau}+\frac{\|u^k-\eta^k\|^2}{\gamma}\nonumber\\	 
		&\hspace{-2cm}+\frac{\|\eta^{k+1}\!-\hat{u}\|^2}{\gamma}
		+\|p^{k+1}\!-x^k\|^2\!\left(\!\frac{1}{\tau}-\frac{1}{2\beta}\!\right)\! + 
		\|\eta^{k+1}\!-u^k\|^2\!\left(\!\frac{1}{\gamma}-\frac{1}{2\delta}\!\right)\nonumber\\
		&\hspace{-2cm}+2\scal{L(p^{k+1}-{x}^k)}{\eta^{k+1}-\hat{u}} 
		-2\scal{L(p^k-x^{k-1})}{\eta^{k}-\hat{u}}	\nonumber\\
		&\hspace{-2cm}-2\|L\|\|p^k-x^{k-1}\|\|\eta^{k+1}-u^k\|\nonumber\\
		&\hspace{-2cm}\geq \frac{\|p^{k+1}-\hat{x}\|^2}{\tau}
		+\frac{\|u^k-\eta^k\|^2}{\gamma}+
		\|\eta^{k+1}-u^k\|^2\left(\frac{1}{\gamma}-\frac{1}{2\delta}-\frac{1}{\nu}\right) 
		\nonumber\\
		&\hspace{-2cm}+
		\left(\frac{1-\alpha}
		{\alpha}\right)\frac{\|x^k-p^k\|^2}{\tau}+\frac{\|\eta^{k+1}-\hat{u}\|^2}{\gamma} 
		+\|p^{k+1}-x^k\|^2\left(\frac{1}{\tau}-\frac{1}{2\beta}\right)
		\nonumber\\
		&\hspace{-2cm}+2\scal{L(p^{k+1}-{x}^k)}{\eta^{k+1}-\hat{u}} 
		-2\scal{L(p^k-x^{k-1})}{\eta^{k}-\hat{u}}	\nonumber\\
		&\hspace{-2cm}-\nu\|L\|^2\|p^k-x^{k-1}\|^2,
	\end{align}
	for every $\nu>0$. If we let 
	$\varepsilon=\left[\left(\frac{1}{\tau}-\frac{1}{2\beta}\right)\left(\frac{1}{\gamma}-\frac{1}{2\delta}\right)-\|L\|^2\right]
	\left(\frac{\beta\tau}{2\beta-\tau}\right)>0,$
	and we choose $\nu=(\frac{1}{\gamma}-\frac{1}{2\delta}-\varepsilon)^{-1}>0$, we have
	$\nu\|L\|^2= 
	(\frac{1}{\tau}-\frac{1}{2\beta})-\nu\varepsilon(\frac{1}{\tau}-\frac{1}{2\beta})$. Hence,
	from \eqref{e:aux1} we have 
	\begin{multline}
		\label{e:fejer}
		\Upsilon_k+\frac{\left\|p^k-\hat{x}\right\|^2}{\tau}\ge\Upsilon_{k+1}+
		\frac{\left\|p^{k+1}-\hat{x}\right\|^2}{\tau}+\left(\frac{1-\alpha}
		{\alpha}\right)\frac{\|x^k-p^k\|^2}{\tau}+\frac{\|u^k-\eta^k\|^2}{\gamma}\\+
		\varepsilon\|\eta^{k+1}-u^k\|^2+\nu\varepsilon\left(\frac{1}{\tau}-\frac{1}{2\beta}\right)\|p^k-x^{k-1}\|^2,
	\end{multline}
	where, for every $k\in\NN$,
	\begin{equation}
		\Upsilon_k=\frac{\|\eta^k-\hat{u}\|^2}{\gamma} + 
		2\scal{L(p^k-x^{k-1})}{\eta^{k}-\hat{u}}+\left(\frac{1}{\tau}-\frac{1}{2\beta}\right)\|p^k-x^{k-1}\|^2.
	\end{equation}
	Note that from \eqref{e:parcond} we have, for every $k\in\NN$,
	\begin{align}
		\Upsilon_k&\ge \frac{\|\eta^k-\hat{u}\|^2}{\gamma} + 
		2\scal{L(p^k-x^{k-1})}{\eta^{k}-\hat{u}}+
		\frac{\|L\|^2}{\left(\frac{1}{\gamma}-\frac{1}{2\delta}\right)}\|p^k-x^{k-1}\|^2\nonumber\\
		&\ge\frac{\|\eta^k-\hat{u}\|^2}{\gamma} + 
		2\scal{L(p^k-x^{k-1})}{\eta^{k}-\hat{u}}+
		\gamma\|L\|^2\|p^k-x^{k-1}\|^2\nonumber\\
		&\ge\frac{1}{\gamma}\|\eta^k-\hat{u}+\gamma L(p^k-x^{k-1})\|^2\ge 0,
	\end{align}
	and, hence, from \eqref{e:fejer} we deduce that 
	$(\Upsilon_k+\|p^k-\hat{x}\|^2/\tau)_{k\in\NN}$ is a F\'ejer sequence. We 
	deduce 
	from \cite[Lemma~5.31]{14.Livre1} that
	$(\eta^k)_{k\in\NN}$ and $(p^k)_{k\in\NN}$ are bounded,
	\begin{equation}
		\label{e:tozero}
		x^k-p^k\to 0,\quad u^k-\eta^k\to 0,\quad \eta^{k+1}-u^k\to 0,\quad \text{and}\quad 
		p^k-x^{k-1}\to 0.
	\end{equation}
	Therefore, there 
	exist weak accumulation points $\bar{x}$ and $\bar{u}$ of the 
	sequences $(p^k)_{k\in\NN}$ and $(\eta^k)_{k\in\NN}$, respectively, 
	say $p^{k_n}\rightharpoonup \bar{x}$ and $\eta^{k_n}\rightharpoonup \bar{u}$ and, from 
	\eqref{e:tozero}, we have $u^{k_n}\rightharpoonup \bar{u}$,
	$u^{k_n+1}\rightharpoonup \bar{u}$,
	$p^{k_n}\rightharpoonup \bar{x}$, $p^{k_n+1}\rightharpoonup \bar{x}$, 
	$x^{k_n-1}\rightharpoonup \bar{x}$ 
	and
	$\bar{x}^{k_n}=x^{k_n}+p^{k_n}-x^{k_n-1}\rightharpoonup \bar{x}$. Since $T$ and $P_V$
	are nonexpansive, $\Id-T$ and $\Id-P_V$ are maximally monotone 
	\cite[Example~20.29]{14.Livre1} and, therefore, they have weak-strong closed graphs
	\cite[Proposition~20.38]{14.Livre1}.
	Hence, it follows from \eqref{e:tozero} that $(\Id-T)p^k\to 0$ and $(\Id-P_V)\eta^k\to 0$
	and, hence, $(\bar{x},\bar{u})\in \Fix T\times V$.
	Moreover,
	\eqref{e:algoinc} can be written equivalently as
	\begin{equation}
		(v^{k_n},w^{k_n})\in (\boldsymbol{M}+\boldsymbol{Q})(p^{k_n+1},\eta^{k_n+1}),
	\end{equation}
	where $\boldsymbol{M}\colon (p,\eta)\mapsto (Ap+L^*\eta)\times(B^{-1}\eta-Lp)$ is 
	maximally 
	monotone \cite[Proposition~2.7(iii)]{15.Siopt1}, $\boldsymbol{Q}\colon 
	(p,\eta)\mapsto 
	(Cp,D^{-1}\eta)$ is $\min\{\beta,\delta\}-$cocoercive, and
	\begin{equation}
		\begin{cases}
			v^{k}:=\frac{x^k-p^{k+1}}{\tau}-L^*(u^{k+1}-\eta^{k+1})+Cp^{k+1}-Cx^k\\
			w^{k}:=\frac{u^k-\eta^{k+1}}{\gamma}+L(x^{k}-p^{k+1}+p^k-x^{k-1})+D^{-1}\eta^{k+1}
			-D^{-1}u^k.
		\end{cases}
	\end{equation}
	It follows from 
	\cite[Corollary~25.5]{14.Livre1} that $\boldsymbol{M}+\boldsymbol{Q}$ is maximally 
	monotone and, since \eqref{e:tozero} and the uniform continuity of $C$, $D$ and $L$
	yields $v^{k_n}\to0$ and $w^{k_n}\to 0$, we deduce from the weak-strong closedness 
	of the 
	graph of $\boldsymbol{M}+\boldsymbol{Q}$ that 
	$(\bar{x},\bar{u})$ is a solution to Problem~\ref{prob:main}, and the result follows. 
	
	\ref{thm:mainiii}: Fix $ k\in\NN$. Since $\rho>0$, $\delta=+\infty$, $\chi=0$, 
	from Theorem~\ref{thm:main}\eqref{thm:maini} we have
	\begin{align}
		\label{e:ineqmainiii}
		\frac{\|x^k-\hat{x}\|^2}{\tau_{k}}+\frac{\|u^k-\hat{u}\|^2}{\gamma_{k}}
		&\geq (2\rho\tau_k+1)\frac{\tau_{k+1}}{\tau_k}\frac{\|p^{k+1}-\hat{x}\|^2}{\tau_{k+1}}
		+\frac{\gamma_{k+1}}{\gamma_k}\frac{\|\eta^{k+1}-\hat{u}\|^2}{\gamma_{k+1}}
		\nonumber\\ 
		&\hspace{-2cm}+2\scal{L(p^{k+1}-{x}^k)}{\eta^{k+1}-\hat{u}} 
		-2\theta_{k-1}\scal{L(p^k-x^{k-1})}{\eta^{k}-\hat{u}}	\nonumber\\
		&\hspace{-2cm}	+\left\|p^{k+1}-x^k\right\|^2\left(\frac{1}{\tau_{k}}-\frac{1}{2\beta}\right)
		-\theta_{k-1}^2\gamma_k\|L\|^2\|p^k-x^{k-1}\|^2,
	\end{align}
	where we have used $2ab\le a^2/\gamma+\gamma b^2$.
	Moreover, it follows from \eqref{e:defparam} that 
	\begin{equation}
		\label{e:paraccel}(\forall k\in\NN)\quad 
		(1+2\rho\tau_{k})\dfrac{\tau_{k+1}}{\tau_{k}} = 
		(1+2\rho\tau_{k})\theta_{k} = 
		\dfrac{1}{\theta_{k}} = \dfrac{\gamma_{k+1}}{\gamma_{k}},
	\end{equation}
	which, combined with \eqref{e:ineqmainiii}, yields
	\begin{align}
		\label{e:ineqmainiii2}
		\frac{\|x^k-\hat{x}\|^2}{\tau_{k}}+\frac{\|u^k-\hat{u}\|^2}{\gamma_{k}}
		&\geq \frac{1}{\theta_k}\left(\frac{\|p^{k+1}-\hat{x}\|^2}{\tau_{k+1}}+
		\frac{\|\eta^{k+1}-\hat{u}\|^2}{\gamma_{k+1}}\right)\nonumber\\
		&\hspace{-2cm}+2\scal{L(p^{k+1}-{x}^k)}{\eta^{k+1}-\hat{u}} 
		-2\theta_{k-1}\scal{L(p^k-x^{k-1})}{\eta^{k}-\hat{u}}\nonumber\\
		&\hspace{-2cm}+\left\|p^{k+1}-x^k\right\|^2\left(\frac{1}{\tau_{k}}-\frac{1}{2\beta}\right) 
		-\theta_{k-1}^2\gamma_k\|L\|^2\|p^k-x^{k-1}\|^2.
	\end{align}
	Now define
	\begin{equation}
		\label{e:defDelta}
		(\forall k\in \mathbb{N})\quad\Delta_{k}=  
		\dfrac{\left\|x^k-\hat{x}\right\|^2}{\tau_{k}}+\dfrac{\left\|u^k-\hat{u}\right\|^2}{\gamma_{k}}
		 .
	\end{equation}
	Dividing \eqref{e:ineqmainiii2} by $\tau_{k}$ and using 
	$\theta_{k}\tau_{k} =\tau_{k+1}$ 
	we 	obtain from the nonexpansivity of $P_V$ and $T$ that
	\begin{align}
		\label{e:auxsxs}
		\dfrac{\Delta_{k}}{\tau_{k}} &\geq \dfrac{\Delta_{k+1}}{\tau_{k+1}} 
		+\frac{2}{\tau_k}\scal{L(p^{k+1}-{x}^k)}{\eta^{k+1}-\hat{u}} 
		-\frac{2}{\tau_{k-1}}\scal{L(p^k-x^{k-1})}{\eta^{k}-\hat{u}}	\nonumber\\
		&\hspace{.5cm}+\frac{\left\|p^{k+1}-x^k\right\|^2}{\tau_k^2}\left(1-\frac{\tau_k}{2\beta}\right)
		-\gamma_k\tau_k\|L\|^2\frac{\|p^k-x^{k-1}\|^2}{\tau_{k-1}^2}.
	\end{align}
	In addition, \eqref{e:parcond} with equality
	reduces to
	\begin{equation}
		\label{e:parcondiii}
		\|L\|^2=
		\left(\frac{1}{\tau_0}-\frac{1}{2\beta}\right)\frac{1}{\gamma_0}\quad
		\Leftrightarrow\quad \gamma_0\tau_0\|L\|^2=
		\left(1-\frac{\tau_0}{2\beta}\right).
	\end{equation}
	Since, for every $k\in\NN\setminus\{0\}$, $\gamma_k\tau_k=\gamma_0\tau_0$ and 
	$\{\tau_{k}\}_{k\in\NN}$ is decreasing (see \eqref{e:defparam}), 
	we have from \eqref{e:parcondiii} that
	\begin{equation}
		\label{e:paraccel2}
		\gamma_{k}\tau_{k}\|L\|^2 =\gamma_{0}\tau_{0}\|L\|^2=
		\left(1-\frac{\tau_0}{2\beta}\right)\le \left(1-\frac{\tau_{k-1}}{2\beta}\right),
	\end{equation}
	and \eqref{e:auxsxs} yields
	\begin{align}
		\label{e:almost}
		\dfrac{\Delta_{k}}{\tau_{k}} &\geq \dfrac{\Delta_{k+1}}{\tau_{k+1}} 
		+\frac{\left\|p^{k+1}-x^k\right\|^2}{\tau_k^2}\left({1}-\frac{\tau_{k}}{2\beta}\right)
		-\frac{\|p^k-x^{k-1}\|^2}{\tau_{k-1}^2}\left({1}-\frac{\tau_{k-1}}{2\beta}\right)
		\nonumber\\
		&\hspace{.5cm}+\frac{2}{\tau_k}\scal{L(p^{k+1}-{x}^k)}{\eta^{k+1}-\hat{u}} 
		-\frac{2}{\tau_{k-1}}\scal{L(p^k-x^{k-1})}{\eta^{k}-\hat{u}}.
	\end{align}	
	Now fix $N\geq 1$. By adding from $k=0$ to $k=N-1$ in \eqref{e:almost},
	using that $\bar{x}^0=x^0$ and defining $p^{0}=x^{0}=:x^{-1}$, we obtain from 
	$u^N=P_V\eta^N$, and 
	$\ran L\subset V$ that 
	\begin{align}
		\dfrac{\Delta_{0}}{\tau_{0}} &\geq \dfrac{\Delta_{N}}{\tau_{N}} + 
		\dfrac{\left\|p^{N}-x^{N-1}\right\|^{2}}{\tau_{N-1}^{2}}\left({1}-\frac{\tau_{N-1}}{2\beta}\right)+\dfrac{2}{\tau_{N-1}}
		\left\langle 
		L(p^{N}-x^{N-1}) \,|\, u^{N} - \hat{u}\right\rangle \nonumber\\
		&\geq \dfrac{\Delta_{N}}{\tau_{N}}  - \frac{\|L\|^{2}}{(1-\frac{\tau_{N-1}}{2\beta})} 
		\left\|u^{N} 
		-\hat{u}\right\|^{2}\nonumber\\
		& = 
		\dfrac{1}{\tau_{N}}\left(\Delta_{N}-\frac{\gamma_N\tau_N\|L\|^{2}}{(1-\frac{\tau_{N-1}}{2\beta})}
		\frac{\left\|u^{N} -\hat{u}\right\|^{2}}{\gamma_N}\right)\geq 
		\frac{\|x^N-\hat{x}\|^2}{\tau_N^2},
		\label{e:ec12}
	\end{align}
	where the last inequality follows from \eqref{e:paraccel2} and \eqref{e:defDelta}.
	Multiplying \eqref{e:ec12} by $\tau_{N}^{2}$ and using $\gamma_{N}\tau_{N} = 
	\gamma_{0}\tau_{0}$ and \eqref{e:parcondiii}, we conclude that
	\begin{equation}
		\left\|x^{N}-\hat{x}\right\|^{2} \leq 
		\tau_{N}^{2}\left(\dfrac{\left\|x^{0}-\hat{x}\right\|^{2}}{\tau_{0}^{2}} + 
		\frac{\|L\|^{2}}{(1-\frac{\tau_0}{2\beta})} 
		\left\|u^{0}-\hat{u}\right\|^{2}\right).
	\end{equation}
	The result follows from $\lim_{N \to \infty} N\rho\tau_{N} 
	= 1$ \cite[Corollary~1]{3.CP}.
	%
	
	\ref{thm:mainiv}: Fix $k\in\NN$. Note that \eqref{e:deftaugamma} yields 
	$\left(\frac{1}{\tau} - \frac{1}{2\beta}\right) 
	\left(\frac{1}{\gamma} - 
	\frac{1}{2\delta}\right) = \left\|L\right\|^{2}$ and, from
	\eqref{e:ec4}, 	$u^{k+1}=P_V\eta^{k+1}$ and $\ran L\subset V$ we have
	\begin{align}
		\label{kec1}
		\frac{\left\|u^k-\hat{u}\right\|^2}{2\gamma} + 
		\frac{\left\|x^k-\hat{x}\right\|^2}{2\tau} 
		&\geq 
		(2\rho\tau+1)\frac{\|p^{k+1}-\hat{x}\|^2}{2\tau}+(2\chi\gamma+1)\frac{\|\eta^{k+1}-\hat{u}\|^2}{2\gamma}
		\nonumber\\ 
		&\hspace{-1cm}+\frac{\big\|p^{k+1}-x^k\big\|^2 }{2}
		\left(\frac{1}{\tau}-\frac{1}{2\beta}\right) + 
		\frac{\|\eta^{k+1}-u^k\|^2}{2}\left(\frac{1}{\gamma}-\frac{1}{2\delta}\right) \nonumber\\
		&\hspace{-1cm}+\scal{L(p^{k+1}-\overline{x}^k)}{\eta^{k+1}-\hat{u}}.
	\end{align}
	Hence, by defining
	\begin{equation}
		\label{e:defOmega}
		(\forall k\in\NN)\quad \Omega_{k} := \left(\chi +\frac{\mu}{4\delta}\right) \|u^{k}-\hat{u} 
		\|^{2} + 
		\left(\rho 
		+\frac{\mu}{4\beta}\right) \|x^{k}-\hat{x} \|^{2},
	\end{equation} 
	multiplying \eqref{kec1} by $\mu$ and using 	\eqref{e:defmualpha}, 
	\eqref{e:deftaugamma},  $u^{k+1} =P_{V} 
	(\eta^{k+1})$, $\ran(L)\subset V$, and the nonexpansivity of $T$ and $P_V$
	we have
	\begin{align}
		\label{kec3}
		\Omega_{k} &\geq \Omega_{k+1} +\mu\rho \big\|p^{k+1}-\hat{x}\big\|^{2} +\mu \chi 
		\big\|\eta^{k+1}-\hat{u}\big\|^{2}+\rho \big\|p^{k+1}-x^{k}\big\|^{2} \nonumber\\
		& \hspace{4cm}
		+\chi 
		\big\|\eta^{k+1}-u^{k}\big\|^{2} +\mu	
		\scal{L(p^{k+1}-\overline{x}^k)}{\eta^{k+1}-\hat{u}}\nonumber\\
		& \geq (1+\alpha)\Omega_{k+1} + \rho \big\|p^{k+1}-x^{k}\big\|^{2} +\chi 
		\big\|u^{k+1}-u^{k}\big\|^{2}\nonumber\\
		&\hspace{4cm}+ \mu \scal{L(p^{k+1}-\overline{x}^k)}{u^{k+1}-\hat{u}}.
	\end{align}
	Moreover, for every $\omega,\lambda>0$ we have
	\begin{align}
		\label{kec4}
		\mu &\scal{L(p^{k+1}-\overline{x}^k)}{u^{k+1}-\hat{u}} \nonumber\\
		&\hspace{1.5cm}= \mu 
		\scal{L(p^{k+1}-x^{k})}{u^{k+1}-\hat{u}} -  \mu\theta
		\scal{L(p^{k}-x^{k-1})}{u^{k+1}-\hat{u}}\nonumber\\
		&\hspace{1.5cm}=\mu \scal{L(p^{k+1}-x^{k})}{u^{k+1}-\hat{u}} - \omega \mu 
		\scal{L(p^{k}-x^{k-1})}{u^{k}-\hat{u}}\nonumber\\
		&\hspace{1.8cm}- \omega\mu \scal{L(p^{k}-x^{k-1})}{u^{k+1}-u^{k}}-(\theta-\omega) 
		\mu 
		\scal{L(p^{k}-x^{k-1})}{u^{k+1}-\hat{u}}\nonumber\\
		&\hspace{1.5cm}\geq \mu \scal{L(p^{k+1}-x^{k})}{u^{k+1}-\hat{u}} - \omega \mu 
		\scal{L(p^{k}-x^{k-1})}{u^{k}-\hat{u}}\nonumber\\
		&\hspace{1.8cm}-\omega\mu\left\|L\right\|\left(\dfrac{\lambda\left\|p^{k}-x^{k-1}\right\|^{2}}{2}
		+ 
		\dfrac{\left\|u^{k+1}-u^{k}\right\|^{2}}{2\lambda}\right)\nonumber\\
		&\hspace{1.8cm}- 
		(\theta-\omega)\mu\left\|L\right\|\left(\dfrac{\lambda\left\|p^{k}-x^{k-1}\right\|^{2}}{2} 
		+ \dfrac{\left\|u^{k+1}-\hat{u}\right\|^{2}}{2\lambda}\right)\nonumber\\
		&\hspace{1.5cm}=\mu \scal{L(p^{k+1}-x^{k})}{u^{k+1}-\hat{u}} - \omega \mu 
		\scal{L(p^{k}-x^{k-1})}{u^{k}-\hat{u}}\nonumber\\
		&\hspace{1.8cm}
		-\mu\theta\lambda\left\|L\right\| \dfrac{\left\|p^{k}-x^{k-1}\right\|^{2}}{2} - 
		\dfrac{\omega\mu\left\|L\right\| \left\|u^{k+1}-u^{k}\right\|^{2}}{2\lambda}\nonumber\\
		&\hspace{1.8cm}-(\theta-\omega)\mu\left\|L\right\|\dfrac{\left\|u^{k+1}-\hat{u}\right\|^{2}}{2\lambda}.
	\end{align}
	By choosing $\lambda=\omega\sqrt{\dfrac{\rho}{\chi}}$, from \eqref{kec3}, 
	\eqref{kec4} and \eqref{e:defmualpha}, we obtain
	\begin{equation}
		\label{kec5}
		\begin{split}
			\Omega_{k}&\geq \dfrac{\Omega_{k+1}}{\omega} + \left(1+\alpha 
			-\dfrac{1}{\omega}\right)\Omega_{k+1}+ \rho \left\|p^{k+1}-x^{k}\right\|^{2}  \\
			& \hspace{.5cm}+\mu \scal{L(p^{k+1}-x^{k})}{u^{k+1}-\hat{u}}- \omega \mu 
			\scal{L(p^{k}-x^{k-1})}{u^{k}-\hat{u}}\\
			&\hspace{.5cm}-\omega\theta\rho 
			\left\|p^{k}-x^{k-1}\right\|^{2}- 
			\left(\dfrac{\theta-\omega}{\omega}\right)\chi\left\|u^{k+1}-\hat{u}\right\|^{2}.
		\end{split}
	\end{equation}
	Since $\theta\in\left](1+\alpha)^{-1},1\right]$, by setting 
	$\omega=\dfrac{1+\theta}{2+\alpha}\in \left[(1+\alpha)^{-1},\theta\right[$, we have 
	$1+\alpha -\dfrac{1}{\omega} = 
	\dfrac{\theta-\omega}{\omega}>0$. Hence,  since \eqref{e:defOmega} yields 
	$\Omega_{k+1}\geq \chi 
	\left\|u^{k+1}-\hat{u}\right\|^{2}$, from 
	\eqref{kec5} and $\theta\le 1$ we have 
	\begin{equation}
		\label{kec6}
		\begin{split}
			\Omega_{k}&\geq \dfrac{\Omega_{k+1}}{\omega} + \rho \left\|p^{k+1}-x^{k}\right\|^{2} 
			-\omega\rho 
			\left\|p^{k}-x^{k-1}\right\|^{2}\\
			& \hspace{.5cm}+\mu \scal{L(p^{k+1}-x^{k})}{u^{k+1}-\hat{u}}- \omega \mu 
			\scal{L(p^{k}-x^{k-1})}{u^{k}-\hat{u}}.
		\end{split}
	\end{equation}
	Moreover, using $p^{0}=x^{0}=:x^{-1}$, multiplying \eqref{kec6} by 
	$\omega^{-k}$ and adding from $k=0$ to $k=N-1$, we conclude from the definition of 
	$\mu$ that 
	\begin{align}
		\label{kec7}
		\Omega_{0}&\geq \omega^{-N}\Omega_{N} +\omega^{-N+1} \rho 
		\left\|p^{N}-x^{N-1}\right\|^{2}+\mu\omega^{-N+1} 
		\scal{L(p^{N}-x^{N-1})}{u^{N}-\hat{u}}\nonumber\\
		&\geq \omega^{-N}\Omega_{N} +\omega^{-N+1} \rho 
		\left\|p^{N}-x^{N-1}\right\|^{2}\nonumber\\
		& \hspace{.5cm}-\mu\omega^{-N+1}\left\|L\right\|\left(\sqrt{\dfrac{\rho}{\chi}} 
		\dfrac{\left\|p^{N}-x^{N-1}\right\|^{2}}{2} + \sqrt{\dfrac{\chi}{\rho}} 
		\dfrac{\left\|u^{N}-\hat{u}\right\|^{2}}{2}\right)\nonumber\\
		& = \omega^{-N}\Omega_{N} - \omega^{-N+1}\chi \left\|u^{N}-\hat{u}\right\|^{2},
	\end{align}
	or, equivalently,
	\begin{multline}
		\label{kec8}
		\omega^{N}\left(\left(\chi + \dfrac{\mu}{4\delta}\right)\left\|u^{0}-\hat{u}\right\|^{2} + 
		\left(\rho + \dfrac{\mu}{4\beta}\right)\left\|x^{0}-\hat{x}\right\|^{2}\right)\\
		\geq \left(\chi(1-\omega) + \dfrac{\mu}{4\delta}\right)\left\|u^{N}-\hat{u}\right\|^{2} + 
		\left(\rho + \dfrac{\mu}{4\beta}\right)\left\|x^{N}-\hat{x}\right\|^{2},
	\end{multline}
	which proves the linear convergence since $\omega <\theta\le 1$.
\end{proof}

\begin{remark}
	\begin{enumerate}
		\item Note that condition \eqref{e:parcond} is weaker than the condition 
		needed in \cite{2.vu}. Indeed, this condition in our case reads
		$2\rho\min\{\beta,\delta\}>1$, where 
		$\rho=\min\left\{{\gamma}^{-1},{\tau}^{-1}\right\}(1-\sqrt{\tau\gamma\|L\|^2})$,
		which implies $2\min\{\delta,\beta\}>\frac{1}{\rho}>\max\{\gamma,\tau\}$,
		\begin{equation}
			\left(1-\frac{\tau}{2\beta}\right)>\sqrt{\tau\gamma \|L\|^2}\quad \text{and}\quad
			\left(1-\frac{\gamma}{2\delta}\right)>\sqrt{\tau\gamma \|L\|^2}.
		\end{equation}
		Thus, by multiplying last expressions we obtain
		$
		\left(1-\frac{\tau}{2\beta}\right)\left(1-\frac{\gamma}{2\delta}\right)>\tau\gamma\|L\|^2,
		$
		which implies \eqref{e:parcond}. Our condition is strictly weaker, as it can be seen 
		in Figure~\ref{fig:regions}, in which we plot the case $\|L\|=1$ and $\delta=\beta=b$, for 
		$b=1$, 
		$b=1/2$ 
		and $b=1/4$. That is, we compare regions 
		\begin{align}
			R_b&=\Menge{(\tau,\gamma)\in[0,2b]\times[0,2b]}
			{\min\left\{\frac{1-\sqrt{\tau\gamma}}{\tau},\frac{1-\sqrt{\tau\gamma}}{\gamma}\right\}>\frac{1}{2b}}\\
			S_b&=\Menge{(\tau,\gamma)\in[0,2b]\times[0,2b]}
			{\left(1-\frac{\tau}{2b}\right)\left(1-\frac{\gamma}{2b}\right)>\tau\gamma}.
		\end{align}
		\begin{figure}[!ht]
			\centering
			\includegraphics[width=0.29\textwidth]{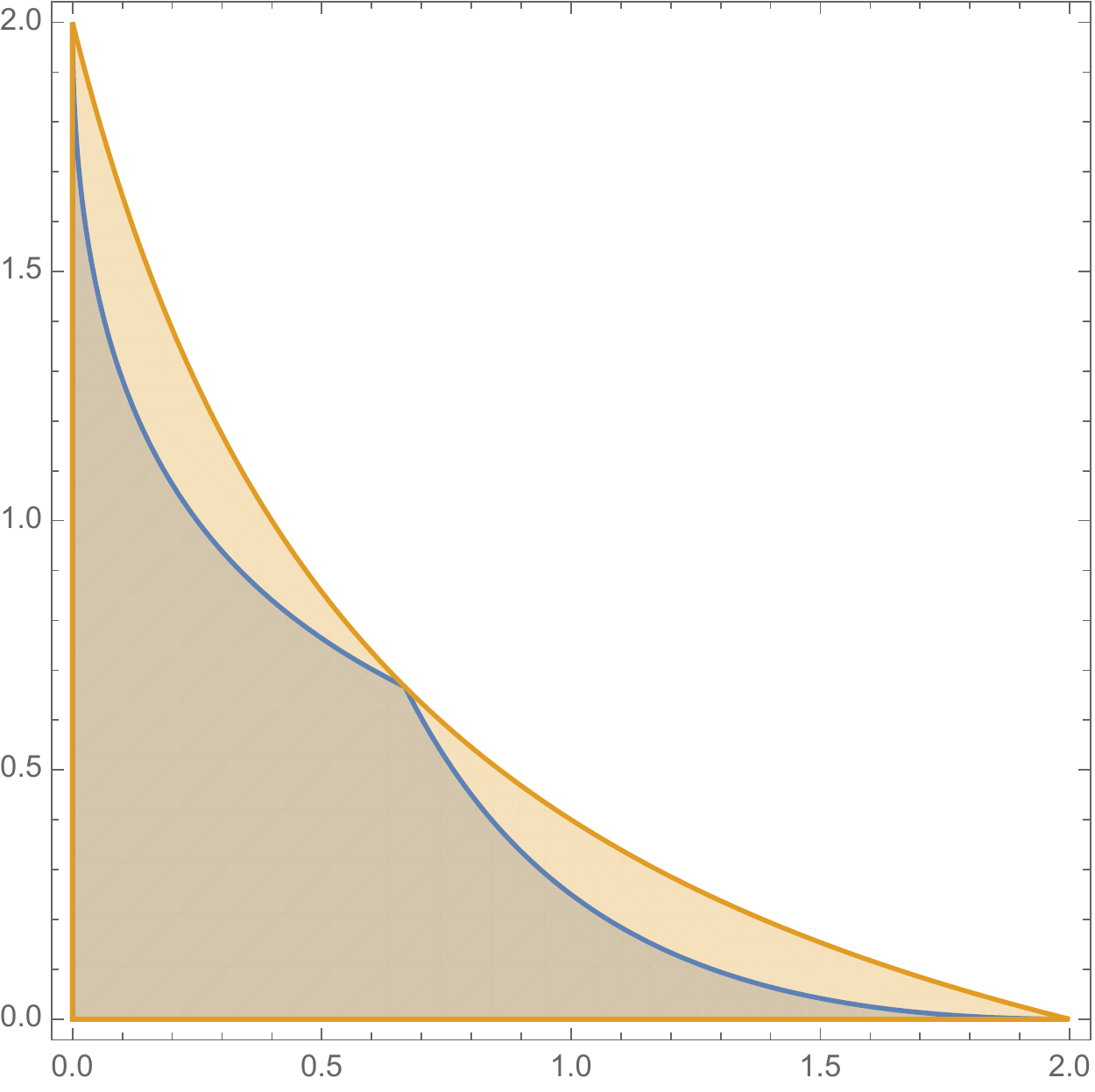}
			\includegraphics[width=0.29\textwidth]{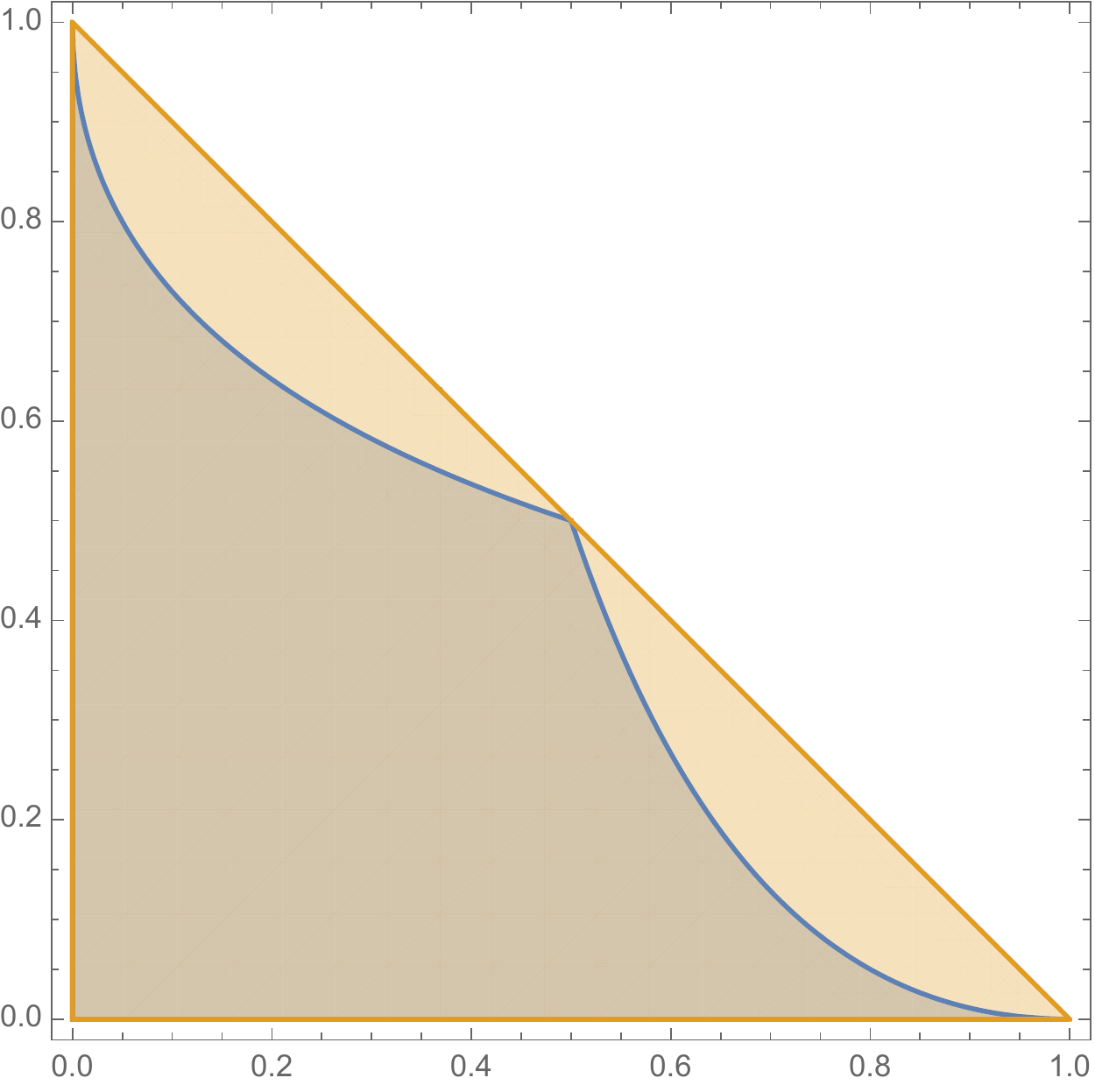}
			\includegraphics[width=0.29\textwidth]{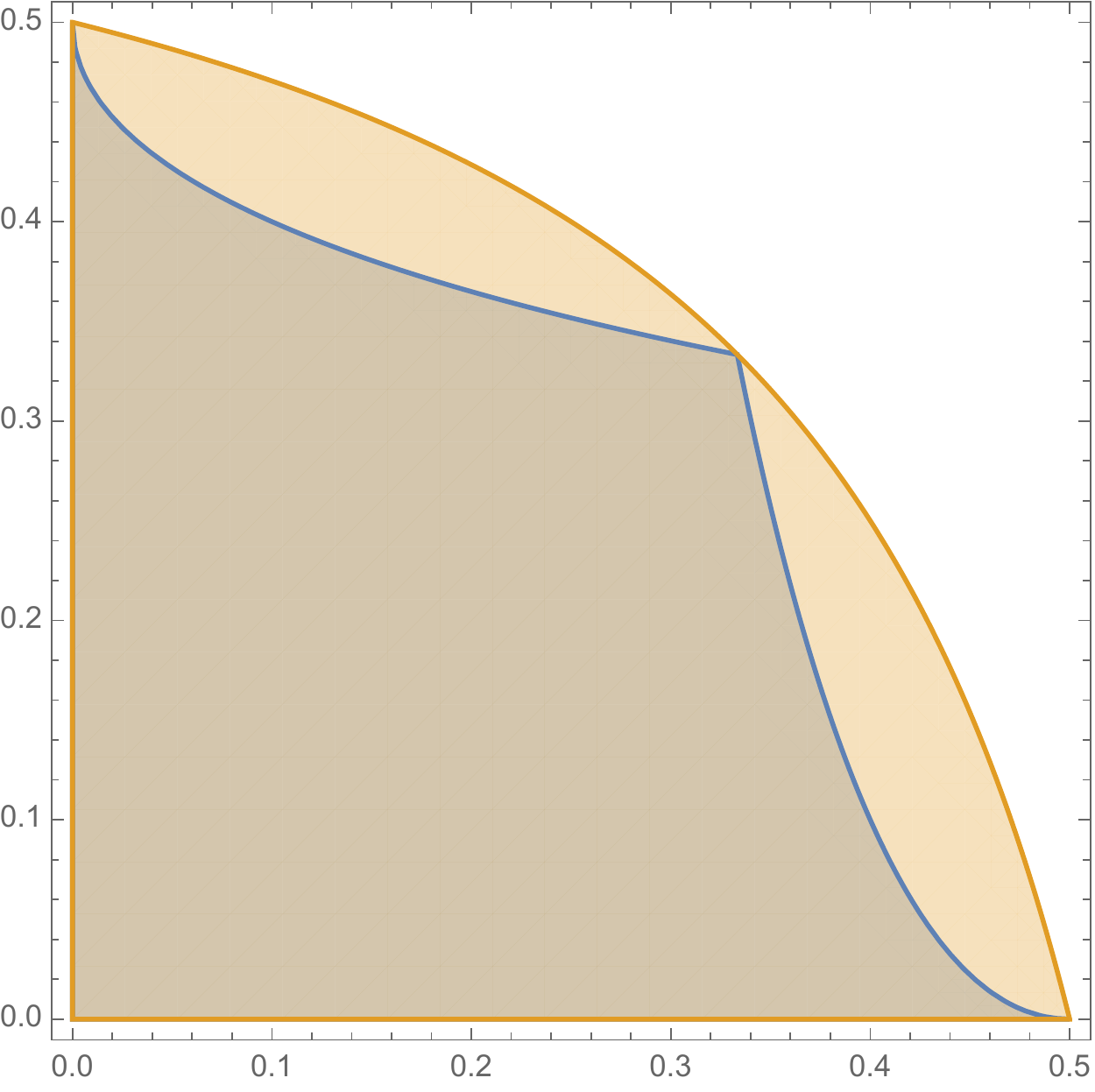}			
			\caption{We plot regions 
				$R_b$ in blue and $S_b$ in orange.
				Left: case $b=1$, Center: case $b=1/2$, Right: case 
				$b=1/4$. Note that in the case $\tau=\gamma$ the regions 
				coincide.}\label{fig:regions}
		\end{figure}
		\item It is not difficult to extend our method by replacing the averaged 
		quasi-nonexpansive operator $T$ by  
		$(\alpha_k)_{k\in\NN}-$averaged quasi-nonexpansive operators $(T_k)_{k\in\NN}$ 
		varying at each iteration 
		and satisfying $\sup_{k\in\NN}\alpha_k<1$.  Indeed, as in \cite{20.Opti04}, we have to 
		assume that $x^k-T_kx^k\to 
		0$ and $x^k\weakly x$ implies 
		$x\in\cap_{k\in\NN}\Fix T_k$, which is satisfied in several cases. In 
		particular, if we set, for every $k\in\NN$, $T_k:=J_{\gamma_kM}(\Id-\gamma_k N)$, 
		where $M\colon \HH\to 2^{\HH}$ is maximally monotone and $N\colon\HH\to\HH$ is 
		$\xi$-cocoercive, if $\gamma_k\in]0,2\xi[$, $T_k$ is $\gamma_k/2\xi-$cocoercive and 
		$\cap_{k\in\NN}\Fix T_k={\rm zer}(M+N)$. Therefore, our method using these operators
		leads to the common solution to ${\rm zer}(M+N)$ and ${\rm zer}(A+L^*\circ B\circ 
		L+C)$. Previous example can also be tackled by Theorem~\ref{thm:main} if we use 
		$\gamma_k\equiv\gamma$ and $T_k\equiv T:=J_{\gamma M}(\Id-\gamma N)$ and we 
		prefer to keep the constant operator case for avoiding additional hypotheses and for 
		the sake of simplicity.
		\item The method proposed in \cite{18.bot14} is an accelerated version of the method 
		proposed in 
		\cite{16.CombPes12} under the assumption that $A+C$ is 
		strongly monotone. Of course, this weaker assumption can also be used in our context,
		but we prefer to keep the statement of Theorem~\ref{thm:main} simpler.
		\item Theorem~\ref{thm:main}\eqref{thm:mainiii} generalizes the acceleration scheme 
		proposed in \cite{3.CP} to monotone inclusions 
		with a priori information. From this 
		results we derive an accelerated version of the methods in \cite{2.vu} in the 
		strongly monotone case when $T=\Id$ and $V=\GG$. These accelerated version, as far 
		as we know, have not been developed in the literature. 
		
		\item In the context of primal-dual problem \eqref{e:primal}-\eqref{e:dual}, 
		\eqref{e:alg2} 
		reduces to
		\begin{align}
			\label{e:algoopt}
			(\forall k\in\NN)\quad
			&\left\lfloor 
			\begin{array}{ll}
				\eta^{k+1}&=\prox_{\gamma_k g^*}(u^k+\gamma_k (L\bar{x}^k-\nabla\ell^*(u^k)))\\
				u^{k+1}&=P_V\,\eta^{k+1}\\
				p^{k+1}&=\prox_{\tau_k f}\left(x^k-\tau_k \left(L^*u^{k+1}+\nabla 
				h(x^k)\right)\right)\\
				x^{k+1}&=T\,p^{k+1}\\
				\bar{x}^{k+1}&=
				x^{k+1}+\theta_k(p^{k+1}-x^{k}),
			\end{array}
			\right.
		\end{align}
		and our conditions on the parameters coincide with 
		\cite{13.Cpock16,13.LorPock15}.
		Without strong convexity of $f$ and $g^*$, we deduce from 
		Theorem~\ref{thm:main}\eqref{thm:mainii} the weak convergence of the sequences 
		generated 
		by \eqref{e:algoopt}, generalizing results in \cite{1.condat,7.BAKS,13.LorPock15}. When 
		$f$ 
		or $g^*$
		is strongly convex, Theorem~\ref{thm:main}\eqref{thm:mainiii} yields an accelerated 
		and projected version of \cite{1.condat}. When $V=\GG$ and $T=\Id$, this result 
		complements the 
		ergodic convergence rates obtained in \cite{13.Cpock16} and generalizes \cite{3.CP}. 
		When 
		$\ell^*=0$, $V=\GG$, $T=\Id$ and $f$ and $g^*$ are strongly convex 
		Theorem~\ref{thm:main}\eqref{thm:mainiv} yields 
		non-ergodic linear convergence of \cite{1.condat}, complementing
		the ergodic linear convergence in \cite{13.Cpock16}.
		The advantage of the algorithm \eqref{e:algoopt} with respect to 
		\cite{1.condat,3.CP} is that primal-dual iterates of the former are forced to be in 
		$X\times 
		V$ when $T=P_X$. This feature leads to a faster 
		algorithm in the context of constrained convex optimization, by choosing $X$ to be
		some of the constraints. This can be observed in the particular instance developed 
		in \cite{7.BAKS} and in Section~\ref{sec6}, in which we provide some numerical 
		simulations.	
	\end{enumerate}
	
\end{remark}

\section{Application to constrained convex optimization}
\label{sec6}
In this section, we explore the advantages of the proposed method in constrained convex 
optimization.
\subsection{Constrained convex optimization problem}
\begin{problem}
	\label{prob:numeric}
	Let $f\in\Gamma_0(\RR^N)$, let $R$ and $S$ be $m\times N$ and $n\times N$ 
	real matrices, respectively, and let
	$c\in\RR^m$ and $d\in \RR^n$.  The problem is to
	\begin{align}
		\label{e:new}
		\min_{x\in\RR^N}&\,f(x)\quad 
		\text{s.t.}\quad  R x=c\quad Sx=d,
	\end{align}
	under the assumption that solutions exist.
\end{problem}	
Note that \eqref{e:new} can be written equivalently as
$
\min_{x\in\RR^N}f(x)+\iota_{\{b\}}( L x),
$
where $L\colon x\mapsto (Rx,Sx)$ and $b=(c,d)\in\RR^{m+n}$. Assume that $0\in{\rm 
	sri}(L(\dom f)-b)$.  Note that, since 
$\prox_{\gamma\iota_{\{b\}}^*}=\Id-\gamma b$ \cite[Proposition~24.8(ix)]{14.Livre1}, the 
method 
proposed in 
\cite[Algorithm~1]{3.CP} in this case reads: given $x^0=\bar{x}^0\in\HH$ and $u^0\in\GG$,
\begin{align}
	(\forall k\in\NN)\quad
	&\left\lfloor 
	\begin{array}{ll}
		u^{k+1}&=u^k+\gamma(L\bar{x}^k-b)\\
		x^{k+1}&=\prox_{\tau f}(x^k-\tau L^*u^{k+1})\\
		\bar{x}^{k+1}&=
		2x^{k+1}-x^{k},
	\end{array}
	\right.
	\label{e:pruebaCP}
\end{align}
where $\gamma\tau\|L\|^2<1$. The constraint is imposed via the Lagrange multiplier update 
in the first step of 
\eqref{e:pruebaCP}. This implies that the primal 
sequence $\{x^k\}_{k\in\NN}$ does not necessarily satisfy any of the constraints.
For ensuring feasibility, we should project onto $L^{-1}b$ by considering the problem 
$\min_{x\in \RR^N}f(x)+\iota_{L^{-1}b}(x)$. However, this is not always possible
since, in several applications, the matrices involved are singular or very bad conditioned
(see discussion in \cite{7.BAKS,13.CEMRACS}). If it is difficult to compute $P_{L^{-1}b}$ but 
we can project onto $R^{-1}c$, we can rewrite \eqref{e:new} as the problem of 
finding $\hat{x}\in R^{-1}c\cap\argmin_{x\in\RR^N}f(x)+\iota_{\{b\}}( L x),$
which is \eqref{e:primal} when $X=R^{-1}c$,
$h=\ell^*=0$, and $g=\iota_{\{b\}}$. Next corollary follows from 
Theorem~\ref{thm:main}, \eqref{e:algoopt} and $P_{X}\colon x\mapsto 
x-R^*(RR^*)^{-1}(Rx-c)$. 
\begin{corollary}
	\label{cor:algo1}
	Let $\gamma>0$ and $\tau>0$ be such that $\gamma\tau\| L\|^2<1$ and 
	let $(x^0,\bar{x}^0,u^0)\in\RR^N\times\RR^N\times\RR^{m+n}$ be 
	such that $x^0=\bar{x}^0$. Consider the routine 
	\begin{align}
		\label{e:prueba}
		(\forall k\in\NN)\quad 
		&\left\lfloor 
		\begin{array}{ll}
			u^{k+1}&=u^k+\gamma( L\bar{x}^k-b)\\
			p^{k+1}&=\prox_{\tau f}(x^k-\tau L^*u^{k+1})\\
			x^{k+1}&=p^{k+1}- R^*( R R^*)^{-1} (Rp^{k+1}-c)\\
			\bar{x}^{k+1}&=x^{k+1}+ p^{k+1}-x^{k}.
		\end{array}
		\right.
	\end{align}
	Then, there exist a solution $\hat{x}$ to Problem~\ref{prob:numeric} and 
	an 
	associated multiplier $\hat{u}$ such that $x^k\to\hat{x}$ 
	and
	$u^k\to\hat{u}$. 
\end{corollary}

\subsection{Numerical experiences}
\label{sec:numeric}
In this section we consider some particular instances of Problem~\ref{prob:numeric}.
We consider the case when $f =\|\cdot\|_1 \in \Gamma_{0}({\mathbb{R}}^{N})$, 
$N=1000$, $\tau=\frac{0.99}{\gamma {\left\|L\right\|}^{2}}$ and the relative error in 
\eqref{e:prueba} is $r_{k}=\sqrt{\frac{\left\|u^{k+1}-u^{k}\right\|^{2} + 
		\left\|x^{k+1}-x^{k}\right\|^{2}}{\left\|u^{k}\right\|^{2} + \left\|x^{k}\right\|^{2}}}$, for every 
$k\in 
\mathbb{N}$. We set $\gamma=10^{-2}$  and 
$(x^0,\bar{x}^0,u^0)=(0,0,0)\in\RR^N\times\RR^N\times\RR^{m+n}$ and, in each test we 
show
the average execution time and the number of average iterations of both 
methods, obtained by considering $20$ random realizations of matrices $R$, $S$ 
and vectors $c \in {\mathbb{R}}^{m}$ and $d\in {\mathbb{R}}^{n}$.
Here PCP and CP  denote the algorithms \eqref{e:prueba} and \eqref{e:pruebaCP}, 
respectively.

\textbf{Test 1.} In Problem~\ref{prob:numeric}, Table~\ref{table:1} 
show the efficiency of CP and PCP for the case $m=1$ and $n=100$. 
\begin{table}[h]
	\centering
	\begin{tabular}{|c|c|c|c|c|c|c|}
		\hline
		\multicolumn{1}{|c|}{\multirow{2}{*}{$m=1$, $n=100$}} & 
		\multicolumn{2}{c|}{\small$e=10^{-4}$} & 
		\multicolumn{2}{c|}{\small$e=5\cdot 10^{-5}$} & 
		\multicolumn{2}{c|}{\small$e=10^{-5}$} \\ \cline{2-7} 
		\multicolumn{1}{|c|}{}                            & iter      & time (s)     & iter      & time (s)      & 
		iter      & time (s)      \\ \hline
		\multicolumn{1}{|c|}{PCP} & \multicolumn{1}{c|}{$9265$} &  
		\multicolumn{1}{c|}{$22.28$}&
		\multicolumn{1}{|c|}{$14570$} &  
		\multicolumn{1}{c|}{$37.02$}& \multicolumn{1}{c|}{$46191$} &  
		\multicolumn{1}{c|}{$116.26$}\\ \hline
		\multicolumn{1}{|c|}{CP} & \multicolumn{1}{c|}{$9732$} &  
		\multicolumn{1}{c|}{$23.04$}& 
		\multicolumn{1}{|c|}{$15718$} &  
		\multicolumn{1}{c|}{$39.21$}& \multicolumn{1}{c|}{$50544$}&  
		\multicolumn{1}{c|}{$125.49$} \\ \hline
		\multicolumn{1}{|c|}{\%improv.} & \multicolumn{1}{c|}{$4.8$}&  
		\multicolumn{1}{c|}{$3.3$} & 
		\multicolumn{1}{|c|}{$7.3$} &  
		\multicolumn{1}{c|}{$5.6$}& \multicolumn{1}{c|}{$8.6$}&  
		\multicolumn{1}{c|}{$7.4$} \\ \hline
	\end{tabular}
	\caption{Average time and number of iterations when $m=1$ for obtaining $r_{k} 
		< e$.}
	\label{table:1}
\end{table}
We see that both algorithms are similar in terms of the execution time and the number of 
iterations, with a small advantage for the PCP algorithm. In addition, by decreasing the 
tolerance $e$, the percentage of improvement, computed as 
$100\cdot(x_{\rm CP}-x_{\rm PCP})/x_{\rm CP}$, slightly increases.

\textbf{Test 2.} In Problem~\ref{prob:numeric}, Table~\ref{table:2} 
show the efficiency of CP and PCP for the case $m=10$ and $n=100$. 
\begin{table}[h]
	\centering
	\caption{Average time and number of iterations when $m=10$ for obtaining  
		$r_{k} < e$.}
	\begin{tabular}{|c|c|c|c|c|c|c|}
		\hline
		\multicolumn{1}{|c|}{\multirow{2}{*}{$m=10$, $n=100$}} & 
		\multicolumn{2}{c|}{\small$e=10^{-4}$} & 
		\multicolumn{2}{c|}{\small$e=5\cdot 10^{-5}$} & 
		\multicolumn{2}{c|}{\small$e=10^{-5}$} \\ \cline{2-7} 
		\multicolumn{1}{|c|}{}                            & iter      & time (s)     & iter      & time (s)      & 
		iter      & time (s)      \\ \hline
		\multicolumn{1}{|c|}{PCP} & \multicolumn{1}{c|}{$6865$} &  
		\multicolumn{1}{c|}{$18.65$}&
		\multicolumn{1}{|c|}{$10229$} &  
		\multicolumn{1}{c|}{$27.86$}& \multicolumn{1}{c|}{$22855$} &  
		\multicolumn{1}{c|}{$65.05$}\\ \hline
		\multicolumn{1}{|c|}{CP} & \multicolumn{1}{c|}{$9280$} &  
		\multicolumn{1}{c|}{$23.72$}& 
		\multicolumn{1}{|c|}{$16033$} &  
		\multicolumn{1}{c|}{$39.13$}& \multicolumn{1}{c|}{$49526$}&  
		\multicolumn{1}{c|}{$129.78$} \\ \hline
		\multicolumn{1}{|c|}{\%improv.} & \multicolumn{1}{c|}{$26.0$}&  
		\multicolumn{1}{c|}{$21.4$} & 
		\multicolumn{1}{|c|}{$36.2$} &  
		\multicolumn{1}{c|}{$28.8$}& \multicolumn{1}{c|}{$53.9$}&  
		\multicolumn{1}{c|}{$49.9$} \\ \hline
	\end{tabular}
	\label{table:2}
\end{table}
In this case, there are clear differences between both algorithms and, as before, PCP is 
more efficient as tolerance decreases. In fact, when tolerance is $10^{-5}$, there is an 
improvement of approximately $50\%$ with respect to the CP in the execution time and the 
number of iterations is less than a half.

\textbf{Test 3.} Finally, in Problem~\ref{prob:numeric}, Table~\ref{table:3} 
show the efficiency of CP and PCP for the case $m=30$ and $n=100$. 
\begin{table}[h]
	\centering
	\caption{Average time and number of iterations when $m=30$ for obtaining  
		$r_{k} < e$.}
	\begin{tabular}{|c|c|c|c|c|c|c|}
		\hline
		\multicolumn{1}{|c|}{\multirow{2}{*}{$m=30$, $n=100$}} & 
		\multicolumn{2}{c|}{\small$e=10^{-4}$} & 
		\multicolumn{2}{c|}{\small$e=5\cdot 10^{-5}$} & 
		\multicolumn{2}{c|}{\small$e=10^{-5}$} \\ \cline{2-7} 
		\multicolumn{1}{|c|}{}                            & iter      & time (s)     & iter      & time (s)      & 
		iter      & time (s)      \\ \hline
		\multicolumn{1}{|c|}{PCP} & \multicolumn{1}{c|}{$5146$} &  
		\multicolumn{1}{c|}{$7.68$}&
		\multicolumn{1}{|c|}{$7143$} &  
		\multicolumn{1}{c|}{$10.67$}& \multicolumn{1}{c|}{$13421$} &  
		\multicolumn{1}{c|}{$19.70$}\\ \hline
		\multicolumn{1}{|c|}{CP} & \multicolumn{1}{c|}{$9941$} &  
		\multicolumn{1}{c|}{$12.93$}& 
		\multicolumn{1}{|c|}{$16438$} &  
		\multicolumn{1}{c|}{$21.37$}& \multicolumn{1}{c|}{$50841$}&  
		\multicolumn{1}{c|}{$64.23$} \\ \hline
		\multicolumn{1}{|c|}{\%improv.} & \multicolumn{1}{c|}{$48.2$}&  
		\multicolumn{1}{c|}{$40.6$} & 
		\multicolumn{1}{|c|}{$56.5$} &  
		\multicolumn{1}{c|}{$50.1$}& \multicolumn{1}{c|}{$73.6$}&  
		\multicolumn{1}{c|}{$69.3$} \\ \hline
	\end{tabular}
	\label{table:3}
\end{table}
We note that the improvement in execution times are considerably higher than in the 
previous cases. For example, in the case of $e = 10 ^{-4}$ the improvement increases by 
approximately $20\%$ with respect to the case $m = 10$ and by approximately $40\%$ in 
the case of $m = 1$. As in the previous cases, if we decrease the tolerance to $10^{-5}$, 
PCP has better efficiency reaching almost $70\%$ improvement with respect to CP. 
Table~\ref{table:7} summarizes the percentage of improvements for each test.
\begin{table}[h]
	\centering
	\caption{Comparison of improvement of average iterations and average times.}
	\begin{tabular}{|c|c|c|c|c|c|c|}
		\hline
		\multicolumn{1}{|c|}{\multirow{2}{*}{\% improv.}} & \multicolumn{2}{c|}{$m=1$} & 
		\multicolumn{2}{c|}{$m=10$} & \multicolumn{2}{c|}{$m=30$} \\ \cline{2-7} 
		\multicolumn{1}{|c|}{}                            & iter      & time (s)     & iter      & time (s)      & 
		iter      & time (s)      \\ \hline
		{\small $e=10^{-4}$}                                      & $4.8$ & $3.3$   & $26.0$       & 
		$21.4$             &  $48.2$         &  $40.6$                                 \\ \hline
		{\small $e=5 \cdot 10^{-5}$ }                                    &     $7.3$      &    $5.6$          &   
		$36.2$        &  $28.8$             &   $56.5$        &   $50.1$            \\ \hline
		{\small $e=10^{-5}$}                                       &  $8.6$         &   $7.4$           &  $53.9$         
		&  $49.9$             &   $73.6$        &     $69.3$          \\ \hline
	\end{tabular}
	\label{table:7}
\end{table}
We observe that for larger values of $m$ we obtain a better relative performance of PCP 
with respect to CP. The larger is $m$, the larger is the proportion of constraints on which we 
project.

\section{Conclusions}
\label{sec5}
In this paper we provide a projected primal-dual method for solving composite monotone 
inclusions with a priori information on solutions. 
We provide acceleration schemes in the presence of strong 
monotonicity and we derive linear convergence in the fully strongly monotone case.
The importance of the a priori information set is illustrated via a numerical example 
in convex optimization with equality constraints, in which the proposed method 
outperforms \cite{3.CP}. 

\textbf{Acknowledgements}
	The authors thank the ``Programa de financiamiento basal'' from CMM--Universidad de 
	Chile and the project DGIP-UTFSM PI-M-18.14 of Universidad T\'ecnica Federico Santa 
	Mar\'ia.

\end{document}